\documentclass[11pt]{amsart}
\usepackage{color,amssymb, thm-restate,hyperref}
\usepackage[normalem]{ulem}
\usepackage[shortlabels]{enumitem}

\usepackage[capitalise]{cleveref}

\usepackage[all]{xy}

\usepackage{graphicx}

\newtheorem{theorem}{Theorem}[section]
\newtheorem{claim}[theorem]{Claim}

\newtheorem{corollary}[theorem]{Corollary}

\newtheorem{lemma}[theorem]{Lemma}

\newtheorem{question}[theorem]{Question}

\newtheorem*{claim*}{Claim}

\theoremstyle{definition}
\newtheorem{remark}[theorem]{Remark}
\newtheorem{definition}[theorem]{Definition}

\newtheorem*{definition*}{Definition}

\newcommand{\N}{\mathbb N}

\newcommand{\Q}{\mathbb Q}

\newcommand{\im}{\operatorname{im}}
\newcommand{\dom}{\operatorname{dom}}
\newcommand{\Sym}{\operatorname{Sym}}

\def\id{\operatorname{id}}

\newcommand{\set}[2]{\{#1:#2\}}

\newcommand{\makeset}[2]{\left\lbrace #1 \;\middle|\;
 \begin{tabular}{@{}l@{}}
   #2
  \end{tabular}
  \right\rbrace}

\author{S. Bardyla, L. Elliott, J. D. Mitchell, and Y. P\'eresse}

\address{S.~Bardyla: University of Vienna, Institute of Mathematics, Austria}
\thanks{The research of the first named author was funded in whole by the Austrian Science Fund FWF [10.55776/ESP399].}
\email{sbardyla@gmail.com}
\address{L. Elliott: University of Manchester, Department of Mathematics, UK}
\email{luna.elliott142857@gmail.com,luna.elliott@manchester.ac.uk}
\address{J. D. Mitchell: University of St Andrews, School of Mathematics and Statistics, Scotland, UK}
\email{jdm3@st-andrews.ac.uk}
\address{Y. P\'eresse: University of Hertfordshire, School of Physics, Engineering and Computer Science, UK}
\email{y.peresse@herts.ac.uk}
\makeatletter
\makeatother

\title{A note on intrinsic topologies of groups}

\begin{document}

\begin{abstract}
We investigate topologies on groups which arise naturally from their algebraic structure, including the Frech\'et-Markov, Hausdorff-Markov, and various kinds of Zariski topologies. 
Answering a question by Dikranjan and Toller, we show that there exists a countable abelian group in which no bounded version of the Zariski topology coincides with the full Zariski topology. Complementing a recent result by Goffer and Greenfeld, we show that on any group with no algebraicity the semigroup Zariski topology is hyperconnected and hence, in many cases, is distinct from the group Zariski topology. 
Finally, we show that on the symmetric groups, the semigroup Hausdorff-Markov topology coincides with the topology of pointwise convergence.  
    \end{abstract}

    \maketitle
%\tableofcontents

% \section{TODO List}
% \begin{enumerate}
% %    \item Write Poor and ask whether for his group condition (2) from Theorem 1.9 holds. Then rewrite it in ZFC.
% %\item define elementary algebraic
% %    \item define group polynomials,
%     \item cite this for zariski on endomorphism monoids with no algebraicity https://www.cambridge.org/core/services/aop-cambridge-core/content/view/AFB95414A365D197E0BD8737A637417A/S0022481223000816a.pdf/on-the-zariski-topology-on-endomorphism-monoids-of-omega-categorical-structures.pdf
% \end{enumerate}

\section{Introduction}
Recall that {\em a group topology} for a group $G$ is any topology $\tau$ on $G$ under which multiplication and inversion are continuous in $(G,\tau)$. The symmetric group $\Sym(X)$ on a set $X$ has a natural group topology, called the {\em topology of pointwise convergence}, which it inherits as a subspace of the Cartesian product $X^X$ under the usual Tychonoff product topology, where every copy of $X$ is given the discrete topology. The topology of pointwise convergence is defined in terms of the underlying set $X$, but it coincides with the group Zariski topology for $\Sym(X)$ which is defined below purely in terms of abstract algebraic properties. In this sense, the topology of pointwise convergence is `intrinsic' to $\Sym(X)$ as an abstract group. 
In this paper, we consider topologies, such as Zariski topologies, which arise from the algebraic properties of any abstract group. In the rest of this section, we state our main results including the necessary definitions and relevant context. Sections $2,3$ and $4$ contain the proofs of our results.

We compose function from left to right and adopt the convention that \(\N=\{0, 1, 2, \ldots\}\).
Let $G$ be a group. A function $f: G\rightarrow G$ is called a {\em group polynomial of degree} $n$ if $(x)f=a_0x^{\varepsilon_1}a_1\dots x^{\varepsilon_n}a_n$, where $a_0,\ldots,a_n\in G$ and $\varepsilon_i\in\{-1,1\}$ for all $i\leq n$. We consider constant functions to be group polynomials of degree $0$. 

\begin{definition}
 For any group $G$ the group Zariski topology $\mathfrak{Z}^{\pm}(G)$ on $G$ is generated by the sets $\{x\in G: (x)f\neq 1_G\}$, where $f$ is a group polynomial.   
\end{definition}

The group Zariski topology was introduced by Markov~\cite{Markov} in order to describe countable nontopologizable groups. Recall that a group $G$ is called {\em nontopologizable} if the only Hausdorff topology on $G$ which makes the group operations continuous is discrete. Markov showed that a countable group $G$ is nontopologizable if and only if the group Zariski topology $\mathfrak{Z}^{\pm}(G)$ is discrete. A topology $\tau$ on a semigroup $S$ is called {\em shift-continuous}, if for every $a\in S$ the shifts $l_a: S\rightarrow S$, $s\mapsto as$ and $r_a: S\rightarrow S$, $s\mapsto sa$ are continuous.
Since for every group $G$ the topology $\mathfrak{Z}^{\pm}(G)$ is shift-continuous (see e.g.~\cite{main}), we have that $\mathfrak{Z}^{\pm}(G)$ is discrete if and only if there exist group polynomials $f_1,\ldots,f_n$ such that $\{1_G\}=\bigcap_{i\leq n}\{x\in G: (x)f_i\neq 1_G\}$. It follows that the discreteness of the group Zariski topology on $G$ is witnessed by the group polynomials of some bounded degree. This perhaps inspired Banakh and Protasov~\cite{BT} to introduce the following weaker versions of the group Zariski topology.
 
\begin{definition}
Let $G$ be a group. For each $n\in\N$ let $\mathfrak{Z}_n^{\pm}(G)$ be the topology on $G$ 
 generated by the sets $\{x\in G: (x)f\neq 1_G\}$, where $f$ is a group polynomial of degree at most $n$.     
\end{definition}

It is clear that for each group $G$, $\mathfrak{Z}_0^{\pm}(G)$ is the anti-discrete topology, $\mathfrak{Z}_n^{\pm}(G)\subseteq \mathfrak{Z}_m^{\pm}(G)$ whenever $n\leq m$, and $\bigcup_{n\in\N}\mathfrak{Z}_n^{\pm}(G)$ forms a basis for $\mathfrak{Z}^{\pm}(G)$.
The arguments above allow us to reformulate Markov's theorem as follows: a countable group $G$ is nontopologizable if and only if $\mathfrak{Z}_n^{\pm}(G)$ is discrete for some $n\in\N$. In particular, for a nontopologizable group $G$ there exists $n\in\N$ such that $\mathfrak{Z}_n^{\pm}(G)=\mathfrak{Z}^{\pm}(G)$. This observation serves as a motivation for the following question posed in~\cite[Question 2.4]{DikranjanToller}. 

\begin{question}[Dikranjan, Toller]\label{Dikr}
    Is it true that for every group \(G\) there exists $n\in\N$ such that $\mathfrak{Z}^{\pm}(G)=\mathfrak{Z}^{\pm}_{n}(G)$? 
\end{question}

%The Zariski topology on groups was intensively investigated in~\cite{???}.

The group Zariski topology has its natural counterpart for semigroups. 
Let $S$ be a semigroup. A function $f: S\rightarrow S$ is called a {\em semigroup polynomial of degree} $n$ if $(x)f=a_0xa_1\dots xa_n$ for some $a_0,\ldots,a_n\in S^1$, where $S^1=S$ if $S$ contains a unit or $S^{1}=S\cup\{1\}$ otherwise. 
Just as in groups, constant functions are semigroup polynomials of degree $0$. %If \(\phi\) is a semigroup polynomial, then we define \(\deg(\phi)\) to be the degree of \(\phi\). 

\begin{definition}
 For any semigroup $S$ the semigroup Zariski topology $\mathfrak{Z}(S)$ on $S$ is generated by the sets $\{x\in S: (x)f\neq (x)g\}$, where $f$ and $g$ are semigroup polynomials.   
\end{definition}

%The following counterpart of Markov's theorem for semigroups was proven by Kotov~\cite{Kotov}: a countable semigroup $S$ admits only discrete Hausdorff semigroup topology if and only if $\mathfrak{Z}(S)$ is discrete. 

The Zariski topologies on groups and semigroups have been investigated in~\cite{Banakh, BT, BMR,Bogo,B,DikranjanToller, DT1, DS1, DS2, DS3, DS4, Jacobs, PS, PS2}. In particular, Zariski topologies play a crucial role in the existence of unique compatible Polish topologies on (semi)groups~\cite{main,main1,Gaughan,Gh,K2,KM,KR,R,Sa}.

The semigroup Zariski topology admits a similar stratification as the group Zariski topology. 

\begin{definition}
Let $S$ be a semigroup. For each $n\in\N$ let $\mathfrak{Z}_n(S)$ be the topology on $S$ 
 generated by the sets $\{x\in S: (x)f\neq (x)g\}$, where $f$ and $g$ are semigroup polynomials of degree at most $n$.     
\end{definition}

It is easy to check that for each group $G$ we have $\mathfrak{Z}(G)\subseteq \mathfrak{Z}^{\pm}(G)$. Let $G$ be a commutative group. Observe that group polynomials on $G$ have the form $(x)f=ax^n$, where $a\in G$ and $n\in \mathbb Z$. Since for each negative integer $n$, $\{x\in G: ax^n\neq 1_G\}=\{x\in G: a^{-1}x^{|n|}\neq 1_G\}$, it follows that $\mathfrak{Z}(G)=\mathfrak{Z}^{\pm}(G)$. Fix any semigroup polynomials $(x)f=ax^n$ and $(x)g=bx^m$ on $G$, where $a,b\in G$ and $n,m\leq k$. Put $(x)h=ab^{-1}x^{n-m}$ and observe that $h$ is a group polynomial of degree $\leq k$. Note that 
$\{x\in G: (x)f\neq (x)g\}=\{x\in G: (x)h\neq 1_G\}$. 
At this point it is easy to see that $\mathfrak{Z}^{\pm}_n(G)=\mathfrak{Z}_n(G)$ for each $n\in\N$.

In this paper we obtain the following answer to Question~\ref{Dikr}. 

\begin{theorem}\label{thcom}
There exists a countable commutative group $G$ such that $\mathfrak{Z}_{n}(G)$ is properly contained in $\mathfrak{Z}_{m}(G)$ for every non-negative integers $n<m$. In particular,
 for all \(n\in \N\) we have 
\[\mathfrak{Z}(G)=\mathfrak{Z}^{\pm}(G)\neq \mathfrak{Z}^{\pm}_{n}(G)=\mathfrak{Z}_{n}(G).\]   
\end{theorem}

%This fact inspired the following question posed in~\cite[Question 2.8]{main}.

The following question appeared in~\cite[Question 2.8]{main}.

\begin{question}[Elliott, Jonu\v{s}as, Mesyan, Mitchell, Morayne, P\'{e}resse]\label{q3}
Does for every group $G$ the semigroup Zariski topology $\mathfrak{Z}(G)$ coincide with the group Zariski topology $\mathfrak{Z}^{\pm}(G)$?
\end{question}

Recently, Goffer and Greenfeld~\cite{GG} answered Question~\ref{q3} in the negative by constructing a countable group $G$ such that $\mathfrak{Z}(G)\neq \mathfrak{Z}^{\pm}(G)$. Although their example is quite technical, as its construction required small cancellation theory.

%Zariski topology on commutative groups was investigated in~\cite{DS2,DS3}. In particular, the following result was proven in~\cite{DS2}. 

%\begin{theorem}[Dikranjan, Shakhmatov]
%For each commutative group $G$, the group Zariski topology $\mathfrak{Z}^{\pm}(G)$ (and thus the semigroup Zariski topology $\mathfrak{Z}(G)$) is Noetherian. In particular, if $G$ is infinite, then $\mathfrak{Z}(G)$ is not Hausdorff.
%\end{theorem}

In this paper, we unveil a whole class of examples of groups with distinct group Zariski and semigroup Zariski topologies. By $\Sym_\omega(X)$ we denote the subgroup of the symmetric group $\Sym(X)$ which consists of elements with finite support.

\begin{theorem}\label{thmain1}
Let $X$ be an infinite set and $G$ be a group such that $\Sym_\omega(X)\subseteq G\subseteq \Sym(X)$. Then $\mathfrak{Z}^{\pm}_{4}(G)=\mathfrak{Z}^{\pm}(G)\neq \mathfrak{Z}(G)$. 
\end{theorem}

Theorem~\ref{thmain1} is based on the following two key theorems. The first one appeared in~\cite{Banakh}.

\begin{theorem}[Banakh, Guran, Protasov]\label{citeB}
Let $X$ be an infinite set and $G$ be a group such that $\Sym_\omega(X)\subseteq G\subseteq \Sym(X)$. Then $\mathfrak{Z}^{\pm}_{4}(G)=\mathfrak{Z}^{\pm}(G)$ is the topology of pointwise convergence. In particular, $\mathfrak{Z}^{\pm}(G)$ is Tychonoff and totally disconnected.
\end{theorem}

The formulating of the second key theorem requires two more definitions. 
\begin{definition}
 A topological space $X$ is called {\em hyperconnected} if the intersection of any two open nonempty subsets of $X$ is nonempty.   
\end{definition}
 Note that no infinite hyperconnected space is Hausdorff. 
\begin{definition}
    For an infinite set $X$ a subgroup $G$ of $\Sym(X)$ is said to have {\em no algebraicity} if for all finite $S\subseteq X$ and $q\in X\setminus S$, the set 
$\{(q)g: g\in G \hbox{ such that } (x)g=x \hbox{ for all }x\in S\}$ is infinite.
\end{definition}
 It is easy to check that the class of groups with no algebraicity includes all dense subgroups of $\Sym(X)$ with respect to the topology of pointwise convergence and, as such, all subgroups of $\Sym(X)$ that contain $\Sym_\omega(X)$; the automorphism group of the following relational structures: the random graph, \((\Q,\leq)\), or the random poset~\cite{PS2}; any group with a Rubin action on an infinite space~\cite{rubin}, in particular the Thompson groups \(F, T, V\), Grigorchuk’s group, and the homeomorphism groups of the following spaces: the Cantor set, the Hilbert cube, or any nonempty manifold of dimension at least~\(1\) (see~\cite{Thompson, Grigorchuk} for background). The second key theorem is the following.

\begin{theorem}\label{hype}
Let $G$ be a subgroup of $\Sym(X)$ with no algebraicity. Then $\mathfrak{Z}(G)$ is hyperconnected.
\end{theorem}

It is easy to see that if a topology $\tau_1$ is contained in a topology $\tau_2$ and $\tau_2$ is hyperconnected, then so is $\tau_1$. Also, note that $\mathfrak{Z}^{\pm}_{3}(G)\subseteq \mathfrak{Z}(G)$ for each group $G$. Thus, Theorem~\ref{hype} implies the following.

\begin{corollary}\label{CorolD}
Let $G$ be a subgroup of $\Sym(X)$ with no algebraicity. Then $\mathfrak{Z}^{\pm}_{3}(G)$ is hyperconnected and thus non-discrete.
\end{corollary}

Corollary~\ref{CorolD} implies a partial solution of the following question posed in~\cite[Question 2.2]{DikranjanToller}.  

\begin{question}[Dikranjan, Toller]
Let $G$ be an infinite group. Is $\mathfrak{Z}_{3}^{\pm}(G)$ always non-discrete?   
\end{question}

Also, Theorem~\ref{hype} and Corollary~\ref{CorolD} establish the sharpness of Theorem~\ref{citeB}.

\begin{corollary}
  Let $X$ be an infinite set and $G$ be a group such that $\Sym_\omega(X)\leq G\leq \Sym(X)$. Then the following dichotomy holds:
  \begin{enumerate}[\rm(i)]
  \item if $n\geq 4$, then $\mathfrak{Z}_n^{\pm}(G)$ is the topology of pointwise convergence and thus is totally disconnected;
  \item if $n\leq 3$, then $\mathfrak{Z}_n^{\pm}(G)$ is hyperconnected.
  \end{enumerate}
\end{corollary}

Note that all the Zariski topologies introduced above depend exclusively on the algebraic structure of a given (semi)group. There are other topologies that reflect mainly algebraic properties of (semi)groups. Namely, the Hausdorff-Markov topology on groups and its semigroup counterparts.

\begin{definition}
Let $G$ be a group. The {\em group Hausdorff-Markov} topology on $G$ is the intersection of all Hausdorff group topologies on $G$.      
\end{definition}

A topology $\tau$ on a semigroup $S$ is called a {\em semigroup topology} if the semigroup operation is continuous in $(S,\tau)$.

\begin{definition}
Let $S$ be a semigroup. The {\em semigroup Hausdorff-Markov} topology $\mathcal {HM}(S)$ is the intersection of all Hausdorff semigroup topologies on $S$. The {\em semigroup Fr\'echet-Markov} topology $\mathcal {FM}(S)$ is the intersection of all $T_1$ semigroup topologies on $S$.     
\end{definition}

%Since each $T_1$ topological group is Tychonoff, the counterparts of Hausdorff-Markov and the Fr\'echet-Markov topologies for groups coincide. Moreover, 
By the celebrated result of Markov~\cite{Markov}, for each countable group $G$, the group Zariski topology $\mathfrak{Z}^{\pm}(G)$ coincides with the group Hausdorff-Markov topology on  $G$. For each semigroup $S$ we have $\mathfrak{Z}(S)\subseteq \mathcal{HM}(S)$, see~\cite{main}. The following question was posed in~\cite[Question 2.7]{main}.

\begin{question}[Elliott, Jonu\v{s}as, Mesyan, Mitchell, Morayne, P\'{e}resse]\label{q4}
Is the semigroup Fr\'echet-Markov topology on a semigroup $S$ always contained in the semigroup Zariski topology on $S$?
\end{question}

Goffer and Greenfeld~\cite{GG} observed that the Shelah group $G$ constructed under CH in~\cite{Shelah} is an example of a group of cardinality $\aleph_1$ such that $\mathcal{FM}(G)$ is discrete, but $\mathfrak{Z}^{\pm}(G)$ is not discrete. In particular, the Shelah group gives a consistent negative answer to Question~\ref{q4}, as well as it  shows that the aforementioned results of Markov cannot be generalized over uncountable groups in ZFC. Recently, Poór and Rinot~\cite{PR} constructed a Shelah-like group $G$ of size $\aleph_1$ in ZFC. They showed that $\mathfrak{Z}^{\pm}(G)$ is not discrete, but every $T_1$ topology that turns $G$ into a topological group is discrete. Their proof however works also for any $T_1$ semigroup topology. Thus we have the following result that settles Question~\ref{q4} in ZFC. 

\begin{theorem}\label{shelah}
There exists a group $G$ of size $\aleph_1$ such that $\mathcal{FM}(G)$ is discrete, but $\mathfrak{Z}^{\pm}(G)$ is not discrete.    
\end{theorem}

Gaughan~\cite{Gaughan} showed that the topology of pointwise convergence coincides with the Hausdorff-Markov topology on $\Sym(X)$. Banakh, Guran and Protasov~\cite{Banakh} showed that the topology of pointwise convergence coincides with the intersection of all Hausdorff shift-continuous topologies on $\Sym(X)$ such that for each $y\in \Sym(X)$ the function $\psi_y:\Sym(X)\rightarrow \Sym(X)$, $x\mapsto xyx^{-1}$ is continuous. 
In this paper, we complement the aforementioned results of Gaughan, and Banakh, Guran and Protasov.

\begin{theorem}\label{paramin}
For every set $X$ the semigroup Hausdorff-Markov topology on $\Sym(X)$ is the topology of pointwise convergence.   
\end{theorem}

\medskip

\section{Separating Zariski topologies on groups}
%Let $G$ be a commutative group.
%As we mentioned in the introduction, $\mathfrak{Z}^{\pm}(G)=\mathfrak{Z}(G)$ and $\mathfrak{Z}^{\pm}_n(G)=\mathfrak{Z}_n(G)$ for each $n\in\N$.
%Hence, to prove Theorem~\ref{thcom} we need to construct a countable commutative group $G$ such that $\mathfrak{Z}^{\pm}_n(G)\neq \mathfrak{Z}^{\pm}(G)$ for all $n\in\N$. 

In this section, we prove Theorems~\ref{thcom} and~\ref{shelah}.

%\begin{theorem}\label{Dikran}
%There exists a countable commutative group $G$ such that $\mathfrak{Z}_{n}(G)$ is properly %contained in $\mathfrak{Z}_{m}(G)$ for every non-negative integers $n<m$.
%\end{theorem}

\begin{proof}[{\bf Proof of Theorem~\ref{thcom}}]
Let $F$ be the free commutative group over a set $\{x_n:n\in \mathbb N\}$. 
For each $k\in \N$ let $N_k$ be the subgroup of $F$ generated by the set $\{x^k_{2n}: n\in\mathbb N\}$. Let $G_k$ be the quotient group $F/N_k$. Finally, let $G$ be the algebraic sum 
$$\bigoplus_{k\in \N} G_k=\big\{f:\N\to \bigcup_{k\in \N}G_k: (k)f\in G_k\hbox { and }  |\{k\in \mathbb N:(k)f\neq 1\}|<\aleph_0\big\}.$$
%Note that \(G_k\) is the direct sum of infinite many copies of \(\mathbb{Z}\times \mathbb{Z}_{k}\), thus \(G\) is the direct sum of infinitely many copies of \(\mathbb{Z}\) and infinitely many copies of all cyclic groups of prime order. 
It is clear that $\mathfrak{Z}_{n}(G)\subseteq \mathfrak{Z}_{m}(G)$ for all $n\leq m$. It is also clear that the topology $\mathfrak{Z}_0(G)$ is anti-discrete and $\mathfrak{Z}_1(G)$ is the cofinite topology. In particular, $\mathfrak{Z}_0(G)\neq \mathfrak{Z}_1(G)$.  
Hence it suffices to show that the topologies $\mathfrak{Z}_{n}(G)$ and 
$\mathfrak{Z}_{m}(G)$ are distinct, whenever $n<m$ are positive integers. Consider the following subset of $G$:
$$T_m=\big\{\phi\in G: (k)\phi=1 \hbox{ if }k\neq m, \hbox{ and }(m)\phi\in \{x_nN_m: n\in \mathbb N\}\big\}.$$

\begin{claim}
The subspace topology on $T_m$ inherited from $(G, \mathfrak{Z}_{n}(G))$ is the cofinite topology.
\end{claim}
\begin{proof}
Fix a semigroup polynomial $(x)f=ax^p$, where $a\in G$ and $p\leq n$. Assume that there exists $x\in T_m$ such that $(x)f=1$. Note that to prove the claim it suffices to show that the set $\{x\in T_m: (x)f=1\}$ is finite. By the choice of $x$, we have \(a=x^{-p}\). Then, by the definition of $T_m$, $(k)a=((k)x)^{-p}=1$ whenever $k\neq m$.
Hence we can restrict our attention to the equation $(m)ax^p=1$ within $G_m$ where $(m)a=hN_m$ for some word $h\in F$. 
Since for every $n\in\N$, $(m)ax_n^p=hN_mx_n^p=hx_n^pN_m$ (multiplication is taken within $F$), we have that $(m)ax_n^p=N_m$ if and only if $hx_n^p\in N_m$. 
For convenience, we identify the group \(G_m\) with the subgroup $\{\phi\in G: (k)\phi=1 \hbox{ for all }k\neq m\}$ of \(G\). Then the element $x_nN_m\in T_m\subseteq G_m$ is a solution to \((x)f=1\) if and only if $hx_n^p\in N_m$.

Suppose that \(x_nN_m\in T_m\) is such that $hx_n^p\in N_m$. If the integer $n$ is even, then $h$ should contain either the letter $x_n$ or the word $x_n^{-p}$, because the number of $x_n$'s in the word $hx_n^p$ should be a multiple of $m$ and $p\leq n<m$. If \(n\) is odd, then \(h\) must contain the word \(x_n^{-p}\), by the definition of $N_m$.
Since $h$ contains only finitely many letters, viewing \(T_m\) as a subset of \(G_m\), we get
\[\big|\makeset{x_nN_m\in T_m}{\(hx_{n}^p\in N_m\)}\big|\leq \big|\makeset{x_n}{\(x_n\) or $x_n^{-1}$ is a letter of \(h\)}\big|<\aleph_0 .\] 
Thus the equation $ax^p=1$ has finitely many solutions in \(T_m\), as required. 
\end{proof}
On the other hand, 
$$\{x\in G:x^m=1\}\cap T_m= \{f\in T_m: (m)f=x_{2n}N_m, n\in\mathbb N\},$$ 
witnessing that the set $\{f\in T_m: (m)f=x_{2n}N_m, n\in\mathbb N\}$ is closed in $T_m$ with respect to the subspace topology inherited from $(G,\mathfrak{Z}_{m}(G))$. In particular, it follows that the subspace topology on $T_m$ inherited from $(G,\mathfrak{Z}_{m}(G))$ is not cofinite.
Hence the topologies $\mathfrak{Z}_{n}(G)$ and $\mathfrak{Z}_{m}(G)$ are distinct, as required. 
\end{proof}

\begin{remark}
Note that the group \(G_k\) from the proof of Theorem~\ref{thcom} is isomorphic to the direct sum of countably infinitely many copies of $\mathbb{Z}\times \mathbb{Z}_{k}$. Thus,
the group $G$ from Theorem~\ref{thcom} is isomorphic to the direct sum of countably infinitely many copies of all cyclic groups.
\end{remark}

Recall that for a positive integer $n$ a group $G$ is called {\em n-Shelah} if for each $A\subseteq G$ such that $|A|=|G|$ we have $A^n=G$.
Recently, Poór and Rinot~\cite{PR} for every regular cardinal $\kappa$ constructed a 10120-Shelah group of size $\kappa^+$ in ZFC. 
Let $G$ be a 10120-Shelah group of size $\aleph_1$ constructed in~\cite{PR}. By~\cite[Lemma 5.11]{PR}, $\mathfrak{Z}^{\pm}(G)$ is not discrete, but every $T_1$ topology that turns $G$ into a topological group is discrete. In the following we explain why their proof works also for any $T_1$ semigroup topology. %Thus we have the following result which answers Question~\ref{q4}. 

%\begin{theorem}\label{shelah}
%There exists a group $G$ of size $\aleph_1$ such that $\mathcal{FM}(G)$ is discrete, but $\mathfrak{Z}^{\pm}(G)$ is not discrete.    
%\end{theorem}
 
\begin{proof}[{\bf Proof of Theorem~\ref{shelah}}] 
Let $G$ be the 10120-Shelah group of size $\aleph_1$ constructed by Poór and Rinot in~\cite{PR}. As we mentioned above, the group Zariski topology $\mathfrak{Z}^{\pm}(G)$ is not discrete. %Since $\mathfrak{Z}(G)\subseteq \mathfrak{Z}^{\pm}(G)$, the semigroup Zariski topology on $G$ is not discrete.
To show that the semigroup Fr\'{e}chet-Markov topology $\mathcal {FM}(G)$ is discrete, consider any $T_1$ semigroup topology $\tau$ on $G$. If $1_G$ has a countable open neighborhood $U$ in $(G,\tau)$, then define $H$ to be a (countable) subgroup of $G$ generated by $U$. By the construction of $G$ (see~\cite{PR}), there exists $g\in G$ such that $H\cap gHg^{-1}=\{1\}$. Since the topology $\tau$ is shift-continuous, the conjugation map $\psi_g: G\rightarrow G$, $x\mapsto gxg^{-1}$ is a homeomorphism of $G$. It follows that $gUg^{-1}$ is an open neighborhood of $1_G$ in $(G,\tau)$. But then $\{1_G\}=U\cap gUg^{-1}\in\tau$, witnessing that $\tau$ is the discrete topology. 

Assume that every open neighborhood of $1_G$ in $(G,\tau)$ is uncountable and thus of cardinality $\aleph_1$. Fix any $g\in G\setminus \{1_G\}$. Since the space $(G,\tau)$ is $T_1$ the set $U=G\setminus\{g\}$ is an open neighborhood of $1_G$. The continuity of the multiplication in $G$ yields the existence of an open neighborhood $V$ of $1_G$ such that $V^{10120}\subseteq U$. Since $G$ is a 10120-Shelah group and $|V|=|G|$, we have $g\in G=V^{10120}\subseteq U$, which contradicts the choice of $U$. The obtained contradiction implies that the semigroup Fr\'echet-Markov topology on $G$ is discrete.
\end{proof}

\section{The semigroup Zariski topology on groups with no algebraicity}

In this section, we are going to prove Theorem~\ref{hype} which together with Theorem~\ref{citeB} establishes Theorem~\ref{thmain1}.  Let $G$ be a group.
Recall that every basic open set $U\in \mathfrak{Z}(G)$ is specified by two finite tuples $(p_0, \dots, p_{n-1})$ and $(q_0, \dots, q_{n-1})$ of semigroup polynomials over $G$ where 
$$U=\set{x \in G}{(x)p_i\neq (x)q_i \text{ for } 0\leq i \leq n-1}.$$
For the purposes of a concise and convenient notation, we will specify such open sets $U$ via pairs of ``ragged matrices" which hold the coefficients of the corresponding semigroup polynomials. 
\begin{definition}
    An {\em $n$-rowed ragged matrix} \(R\) over a semigroup \(S\) is a $n$-tuple $(r_0, \dots, r_{n-1})$ in which every $r_i$ is a $(d_{R,i}+1)$-tuple of elements of $S$, i.e., $r_i =(r_{i,0}, r_{i,1}\ldots r_{i,d_{R,i}})$. We refer to $r_0, \dots,r_{n-1}$ as the {\em rows} of $R$ and write $r_{i,j}$ for the $j$-th entry of $r_i$ (when it exists).
    We will forbid ourselves from denoting a matrix by anything other than an upper-case English letter, so that the corresponding lower-case letter can represent the entries in this fashion.
    Given $x\in S$ we write $Rx$ for the $n$-tuple $(R_0(x), \dots, R_{n-1}(x))$ where
\(R_i(x)= r_{i,0}x\ldots xr_{i, d_{R, i}}.\)
For example
\[\begin{bmatrix}
    a & b & c\\
     d & e &
\end{bmatrix}x=\begin{bmatrix}
    a x b x c\\
    d x e 
\end{bmatrix}.\]
Observe that the semigroup polynomial $R_i(x)$ has degree $d_{R,i}$.
\end{definition}
Given any $n$-rowed ragged matrices $A$ and $B$ over $S$, we let
\[N_{A,B}=\makeset{x\in S}{\(Ax\) and \(Bx\) differ in every coordinate}.\]
%In order to match the notation used until this point, we will break with the usual convention from Linear Algebra and index entries of matrices from $0$ in this section. So, for example, the entries of a \(2\times 2\) matrix \(A\) will be called \(a_{0,0}\), \(a_{0, 1}\), \(a_{1, 0}\) and \(a_{1, 1}\). 
Let $U$ be a basic open set in the semigroup Zariski topology on $S$ given by two tuples of polynomials $(p_0, \dots, p_{n-1})$ and $(q_0, \dots, q_{n-1})$ where $(x)p_i=a_{i,0}xa_{i,1}x\dots xa_{i,d_{A,i}}$ and $(x)q_i=b_{i,0}xb_{i,1}x\dots xb_{i,d_{B,i}}$. Then $U=N_{A,B}$ where $A$ and $B$ are the $n$-rowed ragged matrices with $i$-th rows $(a_{i,0}, \dots, a_{i,d_{A,i}})$ and $(b_{i,0}, \dots, b_{i,d_{B,i}})$, respectively.

In the following, we identify a positive integer $n$ with the set $\{0,\ldots, n-1\}$.

\begin{lemma}\label{SN nice zariski neighbourhood}
Let \(S\) be a cancellative monoid.
The semigroup Zariski topology on \(S\) has a basis consisting of the sets \(N_{A, B}\) where \(A, B\) are ragged matrices over \(S\) such that
\begin{enumerate}
    \item \(d_{A,i} > 0\) or \( d_{B,i}>0\),
      \item \(a_{i, 0} \neq b_{i, 0}\),
\end{enumerate}
for every $i\in k$ where \(k\) is the number of rows of \(A\) and \(B\).
\end{lemma}
\begin{proof}
By the definition of the semigroup Zariski topology, the nonempty sets of the form \(N_{A,B}\), where \(A,B\) are ragged matrices with the same amount of rows, form a basis in $(S,\mathfrak{Z}(S))$. We show that if \(N_{A,B}\neq \varnothing\), then \(N_{A, B}=N_{P, Q}\) for some \(P, Q\) satisfying conditions (1) and (2).

For each pair of ragged matrices \((A, B)\) over \(S\) with \(k\) rows each, let \(t_{A, B}\) be the tuple \[(k, d_{A,0}, \ldots, d_{A,k-1}, d_{B,0}, \ldots, d_{B,k-1}).\] 
For a given pair of $k$-rowed ragged matrices $(A,B)$ such that $N_{A,B}\neq\varnothing$, let $\mathcal{PQ}$ be the set of all pairs \((P, Q)\) of $k$-rowed ragged matrices such that \(t_{P, Q}\) is the minimum element of
\(\makeset{t_{P, Q}}{ \(N_{P, Q}= N_{A, B}\)}\),
where we order these tuples lexicographically.

% We may assume that \(k\in \omega\) is minimal with this property and for all \(i\in k\) there are no \(A'\) and \(B'\) each with \(k\) rows, \(N_{A, B}=N_{A', B'}\), \(A\) and \(A'\) agree in all rows other than \(i\), \(d_i(A')<d_{i}(A)\).
% \\
% Similarly with \(B\).

We claim that each pair $(P,Q)\in\mathcal{PQ}$ satisfies condition (1). 
Suppose to the contrary that \(d_{P,i}=d_{Q,i}=0\) for some \(i\in k\).
As \(N_{P, Q}\) is not empty, we have \(p_{i, 0}\neq q_{i, 0}\).
Thus
\(N_{P, Q} =N_{A, B}\)
where \(A, B\) are obtained from \(P, Q\) by deleting the \(i^{th}\) row, 
which contradicts the minimality of \(k\).

Let us check that there exists a pair $(P',Q')\in\mathcal{PQ}$ that satisfies condition (2). 
Fix any $(P,Q)\in \mathcal{PQ}$. If \(p_{i, 0}\neq q_{i, 0}\) for each $i\in k$, we are done. Otherwise, let $I=\{i\in k: p_{i, 0}= q_{i, 0}\}$. For each $i\in I$ the following three cases require consideration:
\begin{enumerate}[\rm(i)]
\item \(d_{P,i} \neq 0\) and \(d_{Q,i}\neq 0\);
\item \(d_{P,i} = 0\) and \(d_{Q,i}\neq 0\);
\item \(d_{P,i} \neq 0\) and \(d_{Q,i}= 0\).
\end{enumerate}

%\underline{Claim 2}
%If for some \(i\in k\), \(d_{A,i} \neq 0\) and \(d_{B,i}\neq 0\), then \(A[{i, 0}]\neq B[{i, 0}]\).\\
%\underline{Proof of Claim:}\\
%Suppose that \(A[{i, 0}]= B[{i, 0}]\). 
(i) As \(S\) is cancellative for all \(x\in S\),
\[p_{i,0}xp_{i, 1}x \ldots xp_{i, d_{P,i}} \neq q_{i,0}xq_{i, 1}x \ldots xq_{i, d_{Q,i}}\]
if and only if 
\[p_{i, 1}x \ldots xp_{i, d_{P,i}} \neq q_{i, 1}x \ldots xq_{i, d_{Q,i}}.\]
The latter polynomials have smaller degree contradicting the minimality of \(d_{P,i}\) and \(d_{Q,i}\). Hence case~(i) is impossible.

(ii) Let \(f\in S\setminus \{1_{S}\}\).
Then
\[p_{i, 0} \neq Q_{i}(x)\qquad \iff \qquad p_{i, 0}f \neq Q_{i}(x)f.\]

In this case we redefine \(p_{i, 0}:= 
p_{i, 0}f\) and \(q_{i, d_{Q,i}}:= q_{i, d_{Q,i}}f\). 
Note that \(p_{i,0}=p_{i, 0}1_S\neq p_{i, 0}f\) by cancellativity.

Case (iii) is similar to case (ii).

It is clear that the described above process of redefining entries of $P$ and $Q$ gives us a pair $(P',Q')\in\mathcal{PQ}$ that satisfies condition (2).
\end{proof}

For a partial bijection $f$ of a set $X$ let $\dom(f)$ be the domain of $f$ and $\im(f)$ be the image of $f$. 
The following result is folklore. Nevertheless, for the sake of completeness we provide its short proof.

\begin{lemma}[Folklore]
A subgroup $G$ of $\Sym(X)$ has no algebraicity if and only if for each finite partial bijection \(b\) of \(X\) with an extension in \(G\) and for all \(q\in X\backslash\dom(b)\), there is an infinite subset \(A\subseteq X\) such that for each $a\in A$ the partial bijection \(b\cup \{(q, a)\}\) has an extension in \(G\).  
\end{lemma}

\begin{proof}
 ($\Rightarrow$) Fix a finite partial bijection $b$ that has an extension $g_b\in G$ and $q\in X\setminus \dom(b)$. Let $S=\im(b)$ and $q'=(q)g_b$ and $\id_S$ be the identity permutation of $S$. Since $G$ has no algebraicity and $q'\notin S$, there exists an infinite subset $A$ of $X$ such that for each $a\in A$ the partial bijection $\id_S\cup\{(q',a)\}$ extends to an element $g_a$ of $G$. For each $a\in A$ we have $g_bg_a\in G$, $(q)g_bg_a=(q')g_a=a$ and $(g_bg_a){\restriction}_{\dom(b)}=b$.   

 The implication ($\Leftarrow$) follows by choosing $b$ to be a partial identity map. 
\end{proof}

%\begin{definition}
%    If \(X\) is an infinite set, we say that a group \(G\leq \Sym(X)\) has no algebraicity if for all finite \(S\subseteq X\), and \(q\in X\backslash S\), the set
%    \(\makeset{(q)g}{\(g\in G\) fixes \(S\) pointwise}\)
%    is infinite.

%    This is equivalent to saying that for all finite partial bijections \(b\) of \(X\) with an extension in \(G\), and for all \(q\in X\backslash\dom(b)\) there are infinite many \(q'\in X\) such that \(b\cup \{(q, q')\}\) has an extension in  \(G\).
%\end{definition}

%For each set $X$ by $\Sym_\omega(X)$ we denote the subgroup of $\Sym(X)$ consisting of permutations with a finite support.

%\begin{example}\label{noalgexamples}
%Each of the following groups have no algebraicity:
%\begin{enumerate}
%    \item Every dense subgroup of $\Sym(X)$ for any infinite set \(X\). In particular, $\Sym_\omega(X)$ and \(\Sym(X)\). 
 %   \item The automorphism group of the following relational structures: the random graph, \((\Q,\leq)\), or the random poset.
 %   \item Any group with a Rubin action on an infinite space~\cite{rubin}. 
%    Some examples include the Thompson groups \(F, T, V\), Grigorchuk’s group, and the homeomorphism groups of the following spaces: the Cantor set, the Hilbert cube, or any nonempty manifold of dimension at least \(1\) (see \cite{Thompson, Grigorchuk} for background).
%    \item Any group \(G\) such that \(H\leq G\leq \Sym(X)\) where \(H\) has no algebraicity.
%\end{enumerate}    
%\end{example}

We are in a position to prove Theorem~\ref{hype}. We need to show that for each subgroup $G$ of $\Sym(X)$ with no algebraicity, the semigroup Zariski topology $\mathfrak Z(G)$ is hyperconnected.

%\begin{theorem}\label{hyperconnected}
%Let \(X\) be an infinite set and $G$ be a subgroup of $\Sym(X)$ that has no algebraicity. Then the semigroup Zariski topology $\mathcal Z$ on \(G\) is hyperconnected.
%\end{theorem}
\begin{proof}[{\bf Proof of Theorem~\ref{hype}}]
It suffices to check that $U\cap V\neq \varnothing$ for each $U,V$ from the basis $\mathcal B$ of $\mathfrak{Z}(G)$ defined in Lemma~\ref{SN nice zariski neighbourhood}.
Consider \(U=N_{A_1, B_1}\in \mathcal B\) and \(V=N_{A_2, B_2}\in\mathcal B\).
%where \((A_1,B_1)\) and \((A_2,B_2)\) are pairs of ragged matrices such that belong to $\mathcal B$.
We define \(A\) to be the ragged matrix whose rows consist of the rows of \(A_1\) followed by the rows of \(A_2\). The ragged matrix $B$ is defined similarly, by stacking the rows of the matrices $B_1$ and $B_2$. It follows from the definition of \(N_{A, B}\) that 
\[N_{A, B}=N_{A_1, B_1}\cap N_{A_2, B_2}.\]
So it suffices to show that \(N_{A, B}\) is nonempty.

% Let \(N_U, N_V\) be non-empty open sets satisfying the conditions from Lemma~\ref{SN nice zariski neighbourhood}.
% It follows from the conditions on \(N_U, N_V\) that
% \[N:= N_U\cap N_V \subseteq U\cap V\]
% and 
% there exist \(k, m_0, m_1, \ldots, m_k, l_0, l_1, \ldots, l_k \in \omega\) and \(a_{0, 0}, a_{0, 1}, \ldots, a_{0, m_0}, b_{0, 0}, b_{0, 1}, \ldots, b_{0, l_0}, \ldots, a_{k, 0}, a_{k, 1}, \ldots, a_{k, m_0}, b_{k, 0}, b_{k, 1}, \ldots, b_{k, l_0} \in \Sym(\N)\) such that
% \begin{enumerate}
%     \item for all \(i\leq k\), \(a_{i, 0} \neq b_{i, 0}\).
%     \item 
%     \(N=\makeset{x\in \Sym(\N)}{for all \(i\leq k\) we have \(a_{i, 0}xa_{i, 1}x \ldots xa_{i, m_i} \neq b_{i, 0}xb_{i, 1}x \ldots xb_{i, l_i}\)}\).
%     \item there is no \(i\leq k\) such that \(m_i=l_i = 0\).
% \end{enumerate}
% It suffices to show that \(N\neq \varnothing\).
Let \(k\) be the number of rows of \(A\) (or equivalently \(B\)).
By the assumption, for each $i\in k$ the permutations $a_{i, 0}$ and $b_{i, 0}$ are distinct. So, for all \(i\leq k\) there exists \(m_i \in X\) such that \((m_i)a_{i, 0}\neq (m_i)b_{i, 0}\).
Let \(C\) be the set of permutations of \(X\) that occur as entries in either \(A\) or \(B\).
Let $$P :=\{1_X\}\cup  C\cup C^{-1} \cup CC\cup CC^{-1}\cup C^{-1}C\cup C^{-1}C^{-1}.$$ 
Note that \(P\) is finite. 

We compose partial bijections of $X$ as relations and denote this composition by concatenation. In particular, we will do this with elements of $G$.
%If $x$ is a partial bijection of $X$ and $g\in G$ we agree to denote by $gx$ or $xg$ their corresponding composition.
We define a (finite) sequence of partial bijections \(x_j:X \to X\) inductively to satisfy the following properties:
\begin{enumerate}
\item \(|x_j| = j\);
\item for \(l< j\), \(x_l \subseteq x_j\);
    \item if \(i\in k\) and \(j, L, R\in \N\) are such that \(x_j\) is defined and 
    \(m_i\in \operatorname{dom}(a_{i, 0}x_{j} \ldots x_{j} a_{i, L})\cap \operatorname{dom}(b_{i, 0}x_{j} \ldots x_{j} b_{i, R}),\)
    then we have
    \[(m_i) a_{i, 0}x_{j} \ldots x_{j} a_{i, L} \neq (m_i) b_{i, 0}x_{j} \ldots x_{j} b_{i, R};\]
    \item for all \(j\) such that $x_j$ is defined, we have $\{(y)f: y\in\im(x_j), f\in P\}\cap \makeset{m_{p}}{\(p\in k\)} = \varnothing$;
    \item \(x_i\) can be extended to an element of \(G\).
\end{enumerate}
For each $i\in k$, if $x_j$ is defined, then let $p(a,j,i)$ be the maximal number such that 
    $$m_i \in \operatorname{dom}(a_{i, 0}x_{j}\ldots x_{j} a_{i, p(a,j,i)}).$$ Similarly, let $p(b,j,i)$ be the maximal number such that 
    $m_i \in \operatorname{dom}(b_{i, 0}x_{j}\ldots x_{j} b_{i, p(b,j,i)})$.
We start by defining \(x_0:= \varnothing\). 
Note that conditions (1), (2), (4), and (5) are immediately satisfied.
Moreover (3) is vacuously satisfied if \(L\neq 0\) or \(R\neq 0\).
The case that \(L=R=0\) is also satisfied by the choice of \(m_i\). Noteworthy, $p(a,0,i)=p(b,0,i)=0$ for all $i\in k$.

Assume that we constructed partial permutations $x_{i}$, $i\leq n-1$ of $X$ that satisfy conditions (1)--(5).
We will define \(x_n\) from \(x_{n-1}\) in each of the following cases: 
\begin{enumerate}
    \item[($\alpha$)] there exists $j\in k$ such that $\alpha_j:=(m_j)a_{j,0}x_{n-1}\cdots x_{n-1}a_{j,d_{A,j}}$ is not defined;
    \item[($\beta$)] we are not in case ($\alpha$) and there exists $j\in k$ such that $\beta_j:=(m_j)b_{j,0}x_{n-1}\cdots x_{n-1}b_{j,d_{B,j}}$ is not defined.
\end{enumerate}
If neither case ($\alpha$) nor case ($\beta$) occurs, that is for each $j\in k$ the points $\alpha_j$ and $\beta_j$ are defined, then we end the induction. Indeed, by (3), we have that $\alpha_j\neq\beta_j$ for all $j\in k$. By (5), we can extend $x_{m}$ to an arbitrary element $x\in G$. 
Note that for each $j\in k$, $$(m_j)a_{j,0}x\cdots xa_{j,d_{A,j}}\neq (m_j)b_{j,0}x\cdots xb_{j,d_{B,j}}.$$
Hence $x\in N_{A,B}$, as required.

Consider case ($\alpha$). By the assumption, there exists $j\in k$ such that $p(a,n-1,j)<d_{A,j}$. We define \(x_n\) as follows:
\begin{enumerate}[\rm(a)]
    %\item 
    %Let \(p(n-1,j)\in d_{A,j}\) be the maximal number such that 
    %$m_j \in \operatorname{dom}(a_{j, 0}x_{n-1}\ldots x_{n-1} a_{j, p(n-1,j)})$. 
    %\wedge (m_j \not\in \operatorname{dom}(a_{j, 0}x_{n-1}\ldots x_{n-1} a_{j, p + 1}))\]      %\[(d_{B,j} > p)\wedge (m_j \in \operatorname{dom}(b_{j, 0}x_{n-1}\ldots x_{n-1} b_{j, p})) \wedge (m_j \not\in \operatorname{dom}(b_{j, 0}x_{n-1}\ldots x_{n-1} b_{j, p + 1})).\]
    \item Let \(q := (m_j)a_{j, 0}x_{n-1}\ldots x_{n-1} a_{j, p(n-1,j)}\).
    % in the former case and let \(q := (m_j)b_{j, 0}x_{n-1}\ldots x_{n-1} b_{j, p}\) in the latter case and note that \(q\not\in \operatorname{dom}(x_{n-1})\).
    \item Let $T:=\operatorname{dom}(x_{n-1}) \cup \operatorname{img}(x_{n-1}) \cup \{q\}\cup \makeset{m_i}{\(i\in k\)}$. Since the set $(T)P:=\{(t)f: t\in T, f\in P\}$ is finite and \(G\) has no algebraicity, we can choose a point \(q'\in X\setminus (T)P\) such that \(x_n := x_{n-1} \cup \{(q, q')\}\) has an extension in \(G\).
    % This is a partial bijection as \(q'\notin \operatorname{img}(x_{n-1})\subseteq T\subseteq (T)P\) and $q\notin\dom(x_{n-1})$ (as otherwise $p$ would not be maximal).
\end{enumerate}
Note that $p(a,n,i)\geq p(a, n-1,i)$ and $p(b,n,i)\geq p(b, n-1,i)$ for all $i\in k$. Moreover, $p(a,n,j)>p(a,n-1,j)$. This will allow us to finish the induction in finitely many steps.
Let us show that \(x_n\) satisfies conditions (1)--(5).
Conditions (1), (2), (5) are immediate, whereas condition (4) follows from the choice of $q'$ and the inductive assumption. Suppose that condition (3) is not satisfied for some \(i\in k\). Let \(L, R\) be such that 
\begin{equation}\label{not3}\tag{$\ast$}
(m_i)a_{i, 0}x_{n} \ldots x_{n} a_{i, L}= (m_i)b_{i, 0}x_{n} \ldots x_{n} b_{i, R}
\end{equation}
The choice of \(m_{i}\) implies that either \(L>0\) or \(R>0\). We claim that the case \(L=0\) and \(R>0\) is impossible. Indeed, in this case we would have $(m_i)a_{i, 0}= (m_i)b_{i, 0}x_{n} \ldots x_{n} b_{i, R}$, from which follows \((m_i)a_{i, 0}b_{i, R}^{-1}\in \im(x_n)=\im(x_{n-1})\cup \{q'\}\). The latter contradicts either condition (4) of the inductive assumption or the definition of \(q'\), as $a_{i, 0}b_{i, R}^{-1}\in P$. 
Similarly we do not have \(L>0\) and \(R=0\). Thus \(1\leq \min\{L, R\}\).

It follows from condition (3) of the inductive hypothesis that if we replace \(x_n\) with \(x_{n-1}\) in equation~\ref{not3}, then this equality no longer holds. 
Thus there is either a smallest \(l<L\) such that
\[(m_i)a_{i, 0}x_{n} \ldots x_{n} a_{i, l} = q\]
or there is a smallest \(r<R\) such that
\[(m_i)b_{i, 0}x_{n} \ldots x_{n} b_{i, r} = q.\]
We consider the following three subcases of the case ($\alpha$) separately reaching a contradiction in each of them:
\begin{enumerate}[(i)]
    \item \(l\) is defined and \(r\) is not;
    \item \(r\) is defined and \(l\) is not;
    \item \(l,r\) are both defined.
\end{enumerate}

Consider case (i). 
By the minimality of $l$, we have
\[q=(m_i)a_{i, 0}x_{n} \ldots x_{n} a_{i, l}=(m_i)a_{i, 0}x_{n-1} \ldots x_{n-1} a_{i, l}.\]
% or
% \[(m_i)b_{i, 0}x_{n-1} \ldots x_{n-1} b_{i, r} = q.\]

Equation~(\ref{not3}) implies
\(m_i\in \operatorname{dom}(a_{i, 0}x_{n} \ldots x_{n} a_{i, L}) \). Condition (c) implies that $q'\notin (T)P$. It follows that $((q)x_n)f=(q')f\notin \dom(x_n)$ for each $f\in C$.
%\((q)x_n C \cap \operatorname{dom}(x_n) = \varnothing\). 
Since $a_{i,l+1}\in C$, we get $q\notin \dom(x_na_{i,l+1}x_n)$ and, consequently, $q\notin \dom(x_na_{i,l+1}x_na_{i,l+2})$. Thus \(m_i\notin \operatorname{dom}(a_{i, 0}x_{n} \ldots x_{n} a_{i, l+2})\).
Then $L\leq l+1$. Since $l<L$ we get \(L=l+1\).
Thus
\[(m_i)a_{i, 0}x_{n-1} \ldots x_{n-1} a_{i, L-1} = q.\]
% Similarly, if \(r\) is defined then
% \[(m_i)a_{i, 0}x_{n-1} \ldots x_{n-1} a_{i, R-1} = q.\]

Equation~(\ref{not3}) implies
\[(q)x_na_{i,L} =(m_i)b_{i, 0}x_{n} \ldots x_{n} b_{i, R}.\]
As \(r\) is not defined, the equation above yields
\[(q)x_na_{i,L} =(m_i)b_{i, 0}x_{n-1} \ldots x_{n-1} b_{i, R}.\]
So \((q)x_na_{i,L}b_{i, R}^{-1}=(q')a_{i,L}b_{i, R}^{-1}\in \im(x_{n-1})\) and, thus, $q'\in (\im(x_{n-1}))b_{i, R}a_{i,L}^{-1}\subseteq (T)P$, which contradicts the choice of \(q'\). %(recall condition (c) and the definition of $P$).

Case (ii) is similar to case (i).

Consider case (iii). By the minimality of $l$ and $r$, we have
\[q=(m_i)a_{i, 0}x_{n} \ldots x_{n} a_{i, l}=(m_i)a_{i, 0}x_{n-1} \ldots x_{n-1} a_{i, l},\]
and
\[q= (m_i)b_{i, 0}x_{n} \ldots x_{n} b_{i,r}=(m_i)b_{i, 0}x_{n-1} \ldots x_{n-1} b_{i,r}.\]
Hence \[(m_i)a_{i, 0}x_{n-1} \ldots x_{n-1} a_{i, l} =(m_i)b_{i, 0}x_{n-1} \ldots x_{n-1} b_{i,r},\]
which contradicts condition \((3)\) of the induction hypothesis. We have finished case \((\alpha)\). Case \((\beta)\) is dual to case \((\alpha)\).

By the construction, for some $e\leq \sum_{i\in k}(d_{A,i}+d_{B,i})$ we get $x_{e}$ such that $p(a,t,i)=d_{A,i}$ and $p(b,t,i)=d_{B,i}$ for all $i\in k$. In other words, both $(m_i)a_{i,0}x_{e}\cdots x_{e}a_{i,d_{A,i}}$ and $(m_i)b_{i,0}x_{e}\cdots x_{e}a_{i,d_{A,i}}$ are defined for all $i\in k$. Then for $x_e$ neither case ($\alpha$) nor case $(\beta)$ holds.
\end{proof}

It is well known that for every semigroup $S$ the semigroup Zariski topology $\mathfrak{Z}(S)$ is $T_1$ and shift-continuous (see e.g.~\cite{main}). Also, in~\cite{main} it was shown that for each group $G$ inversion is continuous in  $(G,\mathfrak{Z}(G))$. Since infinite hyperconnected spaces are not Hausdorff and each $T_1$ topological group is Tychonoff, Theorem~\ref{hype} implies the following.

\begin{corollary}
Let $G$ be an infinite group with no algebraicity. Then $\mathfrak{Z}(G)$ is not Hausdorff and multiplication is not continuous in $(G,\mathfrak{Z}(G))$. 
\end{corollary}

%\begin{proof}
%Proposition 2.4 from~\cite{main} implies that inversion is always continuous with respect to the Zariski topology, so if Zariski topology is a semigroup topology, then it is a group topology. However, every \(T_1\) group topology is Hausdorff, which would yield a contradiction.
%\end{proof}

\section{Hausdorff semigroup topologies on $\Sym(X)$}

In this section we prove Theorem~\ref{paramin} by using similar techniques as in~\cite{Gaughan}.
Recall that the topology of pointwise convergence on $\Sym(X)$ is generated by the subbase consisting of the sets $U_{x,y}:=\{f\in\Sym(X): (x)f=y\}$, where $x,y\in X$.

\begin{lemma}\label{open iff closed}
Let \(X\) be an infinite set. Let \(\tau\) be a semigroup topology on \(\Sym(X)\) and \(x\in X\). Then 
 \(U_{x,x}\in\tau\) if and only if \(\Sym(X)\setminus U_{x,x}\in\tau\).
\end{lemma}
\begin{proof}
\((\Rightarrow):\) Note that $U_{x,x}$ is a subgroup of $\Sym(X)$. Every open subgroup $G$ of a semitopological group is closed as the complement of \(G\) is the union of cosets of $G$ and hence open.

\((\Leftarrow):\) Assume that $U_{x,x}$ is closed. Let \(y\in X\setminus \{x\}\) be arbitrary and \(g\in \Sym(X)\) be such that \((x)g= y\).
Then the right coset ${U_{x,x}}g$ is $U_{x,y}$. Since the right shift by $g$ is a permutation of $\Sym(X)$, it follows that 
\[V:=(\Sym(X)\setminus U_{x,x})g=\Sym(X)\setminus U_{x,y}=\{f\in \Sym(X): (x)f \neq y\}\]
is open and contains the identity permutation, which we denote by \(1_{X}\). Since $\tau$ is a semigroup topology, there exists an open neighbourhood  \(W\) of \(1_{X}\) such that \(WW\subseteq V\). For every $z\in X$ and $A \subseteq \Sym(X)$ let $$(z)A=\{(z)\rho: \rho\in A\}.$$ Observe that for all \(\rho,\sigma \in W\) we have \((x)\rho\sigma \neq y\) and thus \((x)\rho \neq (y)\sigma^{-1}\). In other words, 
%if for every $z\in \N$ and $V \subseteq \Sym(\N)$ we write $$(z)V=\{(z)\rho: \rho\in V\},$$ then 
$(x)W\cap (y)W^{-1} =\emptyset$.  Thus \((X\setminus (x)W)\cup (X\setminus (y)W^{-1}) =X\). Two cases are possible:
\begin{enumerate}
    \item \(|X\setminus (x)W|= |X|\);
    \item \(|X\setminus (y)W^{-1}|= |X|\).
\end{enumerate}

(1) Note that \(x\in (x)W\) as \(1_X\in W\). Since $|X\setminus (x)W|=|X|$, it is easy to find \(\sigma_x\in \Sym(X)\) such that \((x)\sigma_x=x\) and \((x)W\cap ((x)W)\sigma_x = \{x\}\). We have the following:
\[\{x\}= (x)W\cap (x)W\sigma_x= (x)W \cap (x\sigma_x^{-1})W\sigma_x\supseteq (x)(W\cap \sigma_x^{-1} W \sigma_x).\]
It follows that \(1_X\in W\cap \sigma_x^{-1} W \sigma_x \subseteq U_{x,x}\). Since $\tau$ is a semigroup topology, \(W\cap \sigma_x^{-1} W \sigma_x\in\tau\). Thus \(U_{x,x}\) is a subgroup of $\Sym(X)$ which has a nonempty interior with respect to $\tau$. Hence \(U_{x,x}\) is open as required.

%(2) By Lemma \ref{inversion_continuous_Zariski}, $W^{-1}$ is an open neighbourhood of $1_{\N}$. Thus, replacing $W$ by $W^{-1}$ and $x$ by $y$ in the argument for case (1) above, we conclude that $U_{y,y}$ is open. But $U_{x,x}$ is just the conjugate $hU_{y,y}h^{-1}$ by any $h\in \Sym(\N)$ with $(x)h=y$. Hence $U_{x,x}$ is open, as required.

(2) Note that \(y\in (y)W^{-1}\) as \(1_{X}\in W\). Since $|X\setminus (y)W^{-1}|=|X|$, it is easy to find \(\sigma_y\in \Sym(X)\) such that \((y)W^{-1}\cap ((y)W^{-1})\sigma_y = \{y\}\) and \((y)\sigma_y=y\). We have the following: 
\[\{y\}= (y)W^{-1}\cap (y)W^{-1}\sigma_y= (y)W^{-1} \cap (y\sigma_y^{-1})W^{-1}\sigma_y\supseteq (y)(W^{-1}\cap \sigma_y^{-1} W^{-1} \sigma_y).\]
It follows that \(1_X\in W^{-1}\cap (\sigma_y^{-1}W^{-1}\sigma_y )\subseteq U_{y,y}\) and, as such, $$1_X=1_X^{-1} \in (W^{-1}\cap \sigma_y^{-1} W^{-1} \sigma_y)^{-1} \subseteq U_{y,y}^{-1}.$$
Since $U_{y,y}^{-1}=U_{y,y}$ and $(\sigma_y^{-1} W^{-1} \sigma_y)^{-1}= \sigma_y^{-1} W \sigma_y$, we obtain the following:
$$1_X \in W\cap \sigma_y^{-1} W \sigma_y \subseteq U_{y,y}.$$ 
Hence, being a group with nonempty interior with respect to $\tau$, \(U_{y,y}\) is open as required.
\end{proof}

For a subset $A$ of $X$ let $\Sym(X)_A=\{f\in\Sym(X): (A)f=A\}$.
\begin{lemma}\label{sstab(x,y) closed}
Let \(X\) be an infinite set. 
For all \(x,y\in X\), \(\Sym(X)_{\{x,y\}}\) is a closed subset of \((\Sym(X),\mathfrak Z(\Sym(X)))\).
\end{lemma}
\begin{proof}
For any $x,y\in X$ let \(\phi_{x, y}, \in \Sym(X)\) be such that \((x)\phi_{x, y}=y\), \((y)\phi_{x, y}=x\), and $\phi_{x,y}$ fixes all other points. Observe that for each \(f\in \Sym(X)\) we have $$\{f\in \Sym(X):\phi_{\{x,y\}}f = f\phi_{\{x,y\}}\}=\Sym(X)_{\{x,y\}},$$ as
\[\phi_{x,y}f=f\phi_{x,y} \iff f^{-1}\phi_{x,y}f=\phi_{x,y} \iff \phi_{(x)f,(y)f}=\phi_{x,y} \iff f\in \Sym(X)_{\{x,y\}}.\]
It remains to recall that the set \(\{f\in \Sym(X):\phi_{\{x,y\}}f = f\phi_{\{x,y\}}\}\) is closed in the semigroup Zariski topology on \(\Sym(X)\).
\end{proof}
%{\color{blue}The centralizer of an element of a \(T_1\) paratopological group is not always closed. See thompson's group F with the upper cone topology.}

We say that a subsemigroup \(T\) of a semigroup \(S\) is {\em maximal} if for each $x\in S\setminus T$, $S$ is generated by $T\cup\{x\}$.

\begin{lemma}\label{maximal subsemigroup}
Let \(X\) be an infinite set. \(U_{x,x}\) is a maximal subsemigroup of \(\Sym(X)\) for all  \(x\in X\).
\end{lemma}
\begin{proof}
Let \(f,g\in \Sym(X)\setminus U_{x,x}\). It suffices to show that $g$ belongs to the subsemigroup of $\Sym(X)$ generated by $\{f\}\cup U_{x,x}$. It is easy to see that there exists \(h\in U_{x,x}\) such that \(((x)g)h = (x)f\). Then \((x)ghf^{-1}= x\) and thus \(\phi=ghf^{-1}\in U_{x,x}\). So
\(g=\phi fh^{-1}\in U_{x, x} f U_{x, x}\) 
as required.
\end{proof}

We are in a position to prove Theorem~\ref{paramin}. We need to show that for each set $X$ the semigroup Hausdorff-Markov topology on $\Sym(X)$ is the topology of pointwise convergence.

%\begin{theorem}
%Every Hausdorff semigroup topology on \(\Sym(\N)\) contains the pointwise topology.
%\end{theorem}
\begin{proof}[{\bf Proof of Theorem~\ref{paramin}}]
First let us assume that $X$ is finite.
Then $\Sym(X)$ is finite as well and, as such, every Hausdorff topology on $\Sym(X)$ is discrete. Note that the topology of pointwise convergence on $\Sym(X)$ is also discrete. So in this case we are done. 

Assume that the set $X$ is infinite.
Since the topology of pointwise convergence itself is a Hausdorff semigroup topology, it suffices to show that every Hausdorff semigroup topology on \(\Sym(X)\) contains the topology of pointwise convergence. 
Let \(\tau\) be a Hausdorff semigroup topology on \(\Sym(X)\). Recall that the sets $U_{x,x}=\set{g \in \Sym(X)}{(x)g=x}$, $x\in X$ form a neighbourhood subbasis of the identity in the topology of pointwise convergence. 
Suppose for a contradiction that $\tau$ does not contain the topology of pointwise convergence on $\Sym(X)$. Then there exists \(x\in X\) such that $U_{x,x}$ is not open in $\tau$.
It follows from Lemma~\ref{open iff closed} that \(U_{x,x}\) is not closed in $\tau$. 
Then the closure $\overline{U_{x,x}}$ of $U_{x,x}$ with respect to $\tau$ is a subsemigroup of $\Sym(X)$ that properly contains $U_{x,x}$.
By Lemma~\ref{maximal subsemigroup}, $U_{x,x}$ is a maximal subsemigroup of $\Sym(X)$ and so \(\overline{U_{x,x}}= \Sym(X)\).

Let \(y\in X\setminus \{x\}\).  By Lemma \ref{sstab(x,y) closed}, the set  $\Sym(X)_{\{x,y\}}$ is closed in the semigroup Zariski topology on $\Sym(X)$. Since the Hausdorff semigroup topology $\tau$ contains the semigroup Zariski topology $\mathfrak{Z}(\Sym(X))$, the set $\Sym(X)_{\{x,y\}}$ is closed in $\tau$.
% Hence
% \[\overline{U_{x,x} \cap U_{y,y}}\subseteq Sym(\N)_{\{x,y\}}.\] 
Observe that
$U_{x,x} \cap U_{y,y}=\Sym(X)_{\{x,y\}}\cap U_{x,x}$. It follows that $U_{x,x} \cap U_{y,y}$
is a closed subset of $U_{x,x}$, where the latter is endowed with the subspace topology inherited from $(\Sym(X),\tau)$.
Note that $U_{x,x}$ is isomorphic to $\Sym(X)$. So we apply Lemma~\ref{open iff closed} to $U_{x,x}$ to conclude that \(U_{x,x}\cap U_{y,y}\) is open in $U_{x,x}$.
Therefore there exists \(W\in \tau \) such that 
\[U_{x,x} \cap W =U_{x,x} \cap U_{y,y}.\] 
Then \(V=W\setminus \overline{(U_{x,x} \cap U_{y,y})} \in \tau\) and \(V \cap U_{x,x}= \emptyset\). 
So by the density of \(U_{x,x}\) in $\Sym(X)$, we have \(V= \emptyset\). Taking into account that the set $\Sym(X)_{\{x,y\}}$ is a closed in $(\Sym(X),\tau)$, we get 
\[W\subseteq \overline{(U_{x,x} \cap U_{y,y})}\subseteq \Sym(X)_{\{x,y\}}.\] As \(W\) is a nonempty open set and \(\Sym(X)_{\{x,y\}}\) is a subgroup of $\Sym(X)$, it follows that \(\Sym(X)_{\{x,y\}}\) is open. Let \(z\in X\setminus \{x,y\}\). Let \(\phi\in \Sym(X)\) be such that \((x)\phi=x\), \((y)\phi=z\) and \(\phi^{2}=\operatorname{id}_X\), then 
\[\Sym(X)_{\{x,z\}}=\{\phi f\phi: f\in \Sym(X)_{\{x,y\}}\}=\phi \Sym(X)_{\{x,y\}}\phi\] is also open, as shifts are homeomorphisms in $(\Sym(X),\tau)$.
 It follows that \( \Sym(X)_{\{x,y\}}\cap \Sym(X)_{\{x,z\}}\) is open. As
\[1_{X}\in \Sym(X)_{\{x,y\}}\cap \Sym(X)_{\{x,z\}} \subseteq U_{x, x}\]
and \(U_{x, x}\) is a group, it follows that \(U_{x, x}\) is open, which is a contradiction.
\end{proof}

\end{document}